\documentclass[a4paper,11pt]{amsart}
\usepackage{amssymb,amscd,amsthm,amsmath}
\usepackage{latexsym}
\usepackage{amssymb}
\usepackage{euscript,eufrak,verbatim}
\usepackage[dvips]{graphicx}
\usepackage{verbatim}
\usepackage[final]{epsfig}
\usepackage{ccaption}
\date{\today}

\newtheorem{theorem}{Theorem}
\newtheorem{remark}{Remark}
\newtheorem{lemma}{Lemma}
\newtheorem{prop}{Proposition}

\begin{document}
\title[Statistical Properties of Interval maps]{Statistical Properties of Interval maps with \\ critical points and discontinuities}

\author[H.\ Cui]{Hongfei  Cui}

\address{Wuhan Institute of Physics and Mathematics, The Chinese Academy of Sciences, P.O. Box 71010, Wuhan 430071, China}

\email{cuihongfei05@mails.gucas.ac.cn}




\thanks{{\it Mathematical classification (2000):} 37E05, 28D05, 37D25.}
\thanks{This work was partly supported by the NSFC grant 60534080.}

\begin{abstract}

  We consider dynamical systems given by interval maps with a finite
number of turning points (including critical points,
discontinuities) possibly of different critical orders from two
sides. If such a map $f$ is continuous and piecewise $C^2$,
satisfying negative Schwarzian derivative and some summability
conditions on the growth of derivatives and recurrence along the
turning orbits, then $f$ has finitely many attractors whose union
of basins of attraction has total probability, and each attractor
supports an absolutely continuous invariant probability measure
$\mu$. Over each attractor there exists a renormalization
$(f^m,\mu)$ that is exact, and the rates of mixing (decay of
correlations) are strongly related to the rates of growth of the
derivatives and recurrence along the turning orbits in the
attractors. We also give a sufficient condition for $(f^m,\mu)$ to
satisfy the Central Limit Theorem. In some sense, we give a fairly
complete global picture of the dynamics of such maps. Similarly,
we can get similar statistical properties for interval maps with
critical points and discontinuities under some more assumptions.

\end{abstract}

\maketitle

\section{Introduction and statement of results}

\subsection{Introduction}
In the last three decades, many results on statistical properties
were obtained for iterations of one dimensional maps, for example
the existence of absolutely continuous invariant probability
measures (acip for short), the decay of correlations and the
Central Limit Theorem. Various conditions have been shown to
guarantee the existence of acip and corresponding statistical
properties. In the area of interval maps with turning points
(including critical points, discontinuities), these results
generally require more restriction on the turning points, for
example the critical orders of each turning point from two sides
are equal in the continuous cases, and  positive Lyapunov
exponents at the critical
 values and the recurrence of turning point are sub-exponential in the discontinuous cases.

Our aim in this paper is to obtain the same conclusion but relax
as much as possible the conditions on the orbit of the critical
points, to include in particular cases in which the growth of
derivatives may be sub-exponential and/or the recurrence of the
turning points exponential, or the critical orders of turning
points from both sides do not equal. More precisely, we will show
the existence and finiteness of the number of acip for a general
map under two summability conditions on the growth of derivatives
and recurrence along the turning points, and study its statistical
properties such as decay of correlations and the Central Limit
Theorem. In these processes, we give a complete global picture of
the dynamics from a probabilistic perspective.

\subsection{Dynamical assumption}

We now give the precise statement of our assumption. Let
$\mathcal{A}$ denote the class of the interval map satisfying the
conditions formulated in Subsections 1.2.1-1.2.2 below. Denote
$\mathcal{A}_1$ as continuous maps in $\mathcal{A}$ and
$\mathcal{A}_2$ as discontinuous maps in $\mathcal{A}$
respectively.

\subsubsection{Critical set}

Let $M$ be a compact interval $[0,1]$ and $f:M\rightarrow M$ be a
piecewise $C^2$ map. This means that there exists a finite set
$\mathcal{C}$ such that $f$ is $C^2$ and a diffeomorphism on each
component of $M\setminus \mathcal{C}$, and $f$ admits a continuous
extension to the boundary so that both the left and the right
limits $f(c{\pm})= \lim_{x \rightarrow c\pm} f(x)$ exist. We
regard each $c\in \mathcal{C}$ as two points: $c+$, $c-$, the
concrete values depend on the corresponding one-side
neighborhoods. We assume that each $c\in \mathcal{C}$ has a
one-side critical order $l(c\pm)\in [1, \infty)$, this means that
$$
 |Df(x)|\approx
{|x-c|}^{l( c\pm)-1}, \ \ |f(x)-f(c\pm )|\approx{|x-c|}^{l(
c\pm)}, \ and\ |D^2f(x)|\approx {|x-c|}^{l(c\pm)-2}
$$
for $x$ in the corresponding one-side neighborhood of $c$,
 where we say $f\approx g$ if the ratio $f/g$ is bounded
above  and below uniformly in its domain. When we use the notion
$l(c)$, it may be either $l(c+)$ or $l(c-)$. If $l(c) = 1$, $c$ is
a bound derivative point, and if $l(c)> 1$ we say that $c$ is a
critical point. Note that $c$ may be a critical point on one side
and is a bound derivative point on the other side. When there is
no possibility of confusion, each point $c \in \mathcal{C}$ will
be called a critical point without distinguishing whether $c$ is
really a critical point with $l(c)>1$, or $c$ is a bounded
derivative point with $l(c)=1$.


We also assume that $f$ is with negative Schwarzian derivative
outside of $\mathcal{C}$, i.e., ${|Df|}^{-\frac{1}{2}}$ is a
convex function on each component of $M\setminus \mathcal{C}$. In
particular, if $f$ is continuous, and the critical orders with
$l(c)>1$ are equal from two sides for each critical point, we can
get rid of this assumption (but need to add a natural topological
assumption that all period points are hyperbolic repelling) by a
result of Kozlovski \cite{ko} (generalized to the multimodal
setting by van Strien and Vargas \cite{sv}).

\subsubsection{Summability conditions}

We suppose $f$ satisfies the following summability conditions
along the critical orbits. The first summability condition is

\begin{equation}\label{the first summability1}
\sum_{n=1}^{\infty} { \Big{(} \frac {
{|f^n(c)-\tilde{c}|}^{l(\tilde{c})} }
{{|f^n(c)-\tilde{c}|}^{l(c)}|Df^n(f(c))|} \Big{)} }^{1/(2l(c)-1)}
< \infty , \ \ \forall c \in \mathcal{C},
\end{equation}
where $\tilde{c}$ is the critical point closest to $f^n(c)$, and
 $l(c)$, $l(\tilde{c})$ depend on the corresponding one-side neighborhoods, and the second summability condition is
 \begin{equation}\label{the second summability}
\sum_{n=1}^{\infty} { \big{(} \frac{1}{|Df^n(f(c))|} \big{)}
}^{1/l(c)} < \infty, \ \ \forall c \in \mathcal{C}.
\end{equation}

One of the most simplest example satisfying the above conditions
is the contracting Lorenz maps considered in \cite{me} and
\cite{ro}, which motivated by the study of the return map of the
Lorenz equations near classical parameter values. Notice that
above summability conditions are satisfied if the derivative is
growing exponentially fast and the recurrence is not faster than
exponential in the sense that for each critical point $c\in
\mathcal{C}$, $|Df^n(f(c))|>{\lambda}^n$, for some $\lambda>1$,
and $|f^{n-1}(f(c))-\mathcal{C}|>{e}^{-\alpha
  n}$ for some $\alpha$ small enough,  for all $n\geq 1$.

{\remark According to the first summability condition, critical
points of $f\in \mathcal{A}$ are not on the forward orbits of the
critical set, i.e., $\mathcal{C}\cap\cup_{n\geq
1}f^n(\mathcal{C})=\emptyset$. It is easy to see that if all of
the critical orders are equal, one can get rid of the recurrence
condition containing in the first summability condition. }

\subsection{Statement of results}

In a previous paper, we have shown the following theorem under
some weaker summability conditions,
\begin{theorem}\label{existence}\cite{cd}
If $f$ satisfies assumption in Subsection 1.2.1 and summability
condition (\ref{the second summability}), and the following
summability condition
$$
\sum_{n=1}^{\infty} { \Big{(} \frac {
{|f^n(c)-\tilde{c}|}^{l(\tilde{c})} }
{{|f^n(c)-\tilde{c}|}^{l(c)}|Df^n(f(c))|} \Big{)} }^{1/l(c)} <
\infty, \forall c \in \mathcal{C},
$$
then $f$ admits an acip. Furthermore, if $l_{\max}>1$, then its
density is in $L^p$ for all $1\leq p
<\frac{l_{\max}}{l_{\max}-1}$, where $l_{\max}$ is the maximum of
the orders of the critical points.
\end{theorem}

In this paper, we will consider the general properties of acip,
including the finiteness of acip, the support of each acip and its
properties (decay of correlations, Central Limit theorem). In
general, if $f$ has many turning points, the acip $\mu$ need not
to be unique and not Lebesgue ergodic (note that unimodal maps
with negative Schwarzian derivative and equal critical orders from
both sides are ergodic
 with respect to Lebsegue measure). If $f\in \mathcal{A}_1$ ,
 then we can choose the minimal cycle (see Section 3) of
 $f$, show that there exists a unique acip $\mu$ support on the minimal
 cycle and the renormalization $f^m$ of $f$ corresponding the
 minimal cycle is exact, hence $(f^m,\mu)$ is mixing and ergodic.
 So it is natural to estimate the rates of mixing,
 quantified through the correlation functions
 $$
C_n(f^m,\mu) : = |\int (\varphi\circ f^m)\psi d\mu - \int\varphi
d\mu \int\psi d\mu|,
 $$
where $\varphi$ and $\psi$ are respectively bounded and
H$\ddot{o}$lder continuous functions on the minimal cycle.

Another important characterization of $\mu$ is by Central Limit
Theorem, which describes the oscillations of finite averages
$\frac{1}{n}\sum_{j=0}^{n-1}\varphi(f^j(x))$ around the expect
value $\int \varphi d\mu$. We say that $(f^m, \mu)$ satisfies the
Central Limit Theorem if given a H\"{o}lder continuous functions
$\varphi$ which is not coboundary ($\varphi \neq \psi\circ f-\psi
$ for any $\psi$), then there exists $\sigma >0$ such that for any
interval $I\subset R$,
$$
\mu \Big{\{} x\in M; \big{(}\frac{1}{\sqrt{n}} \sum_{j=0}^{n-1}
\varphi \circ(f^j(x))- \int \varphi d\mu \big{)} \in I \Big{\}}
\rightarrow \frac{1}{\sigma \sqrt{2\pi}}
\int_{I}e^{-\frac{t^2}{2{\sigma}^2}} dt, \ as \ n \to \infty.
$$

We shall state the results in this paper. Write
\begin{equation}
\gamma_n(c):= \min \Big{\{} { \Big{(} \frac
{{|f^n(c)-\tilde{c}|}^{l(\tilde{c})} }
{{|f^n(c)-\tilde{c}|}^{l(c)}|Df^n(f(c))|} \Big{)} }^{1/(2l(c)-1)},
\  \frac{1}{2} \Big{\}}  \ \ \ \forall c \in \mathcal{C},
 \end{equation}
\begin{equation}
d_n(c):=\min_{i<n} {(\frac{\gamma_i(c)}{|Df^i(f(c))|})}^{1/l(c)}
{|f^i(c)-\tilde{c}|}^{l(\tilde{c})/l(c)}, \ \forall c \in
\mathcal{C}.
\end{equation}
By the first summability condition and elementary calculations, we
have
 $d_n(c)\leq \gamma_{n-1}(c)$.

\begin{theorem}
Let $f\in \mathcal{A}_1$, then there exists at least one and at
most $\sharp \mathcal{C}$ different ergodic acips $\mu_i, 1\leq
i\leq N$. These measures are support on minimal cycles of
intervals with pairwise disjoint interiors. Over each such cycles
there exists a renormalization $(f^{m_i}, \mu_i)$ that is exact.
For a fixed minimal cycle with critical point $\mathcal{C}_c$, the
corresponding renormalization $(f^{m}, \mu)$ is mixing with the
following rates:

Polynomial case: If $d_n(c)\leq C n^{-\alpha}$, $\alpha>1$, $C>0$
for all $c\in \mathcal{C}_c$ and $n\geq 1$, then for each
$\tilde{\alpha}<\alpha-1$, there exists $\tilde{C}>0$ such that
$$
C_n(f^m,\mu)  \leq \tilde{C} n^{-\tilde{\alpha}}.
$$

Stretched exponential case: If $\gamma_n(c)\leq Ce^{-\beta
n^{\alpha}}$, $C>0$, $\alpha\in (0,1)$, $\beta>0$ for all $c\in
\mathcal{C}_c$ and $n\geq 1$, then for each $\tilde{\alpha}\in
(0,\alpha)$ there exist $\tilde{\beta}, \tilde{C}>0$ such that
$$
C_n(f^m,\mu)  \leq \tilde{C} e^{-\tilde{\beta}n^{\tilde{\alpha}}}.
$$

Exponential case: If $\gamma_n(c)\leq Ce^{-\beta n}$, $\beta>0$
for all $c\in \mathcal{C}_c$ and $n\geq 1$, then there exist
$\tilde{\beta}, \tilde{C}>0$ such that
$$
C_n(f^m,\mu)  \leq \tilde{C} e^{-\tilde{\beta}n}.
$$

If $d_n(c)\leq C n^{-\alpha}$, $C>0$, $\alpha>2$ for $c\in
\mathcal{C}_c$ and $n\geq 1$, then  $(f^m,\mu)$ satisfies the
Central Limit theorem. If $l_{\max}>1$ for critical point in the
minimal cycle, then the density of $\mu$ is $L^p$ for all $1\leq
p<\frac{\l_{\max}}{l_{\max}-1}$. The union of the basin $B(\mu_i)$
has full probability measure in the interval $M$.
\end{theorem}

For a map $f\in \mathcal{A}_2$, we suppose $f$ satisfies the
following

{\bf Density of preimages}: There exists $c\in \mathcal{C}$ whose
preimages are dense in some $J^*\in M$, where $J^*$ is a union of
intervals and satisfies $f(J^*)= J^*$.

{\remark The interval $J^*$ plays the same role to the notation of
the minimal cycle in the continuous case.}

\begin{theorem}
Let $f\in \mathcal{A}_2$ and satisfies the assumption of density
of preimages. Then $f$ has an acip $\mu$, and there exists an
integer $k>1$ such that $(f^k, \mu)$ is exact. For H$\ddot{o}$lder
continuous functions $\varphi$, $\psi$, $(f^k, \mu)$ is mixing
with the following rates:

Polynomial case: If $d_n(c)\leq C n^{-\alpha}$, $\alpha>1$, $C>0$
for $c\in J^*$ and $n\geq 1$, then for each
$\tilde{\alpha}<\alpha-1$, there exists $\tilde{C}>0$ such that
$$
C_n(f^k,\mu)  \leq \tilde{C} n^{-\tilde{\alpha}}.
$$

Stretched exponential case: If $\gamma_n(c)\leq Ce^{-\beta
n^{\alpha}}$, $C>0$, $\alpha\in (0,1)$, $\beta>0$ for  $c\in J^*$
and $n\geq 1$, then for each $\tilde{\alpha}\in (0,\alpha)$ there
exist $\tilde{\beta}, \tilde{C}>0$ such that
$$
C_n(f^k,\mu)  \leq \tilde{C} e^{-\tilde{\beta}n^{\tilde{\alpha}}}.
$$

Exponential case: If $\gamma_n(c)\leq Ce^{-\beta n}$, $C>0$,
$\beta>0$ for $c\in J^*$ and $n\geq 1$, then there exist
$\tilde{\beta}, \tilde{C}>0$ such that
$$
C_n(f^k,\mu)  \leq \tilde{C} e^{-\tilde{\beta}n}.
$$

If $d_n(c)\leq C n^{-\alpha}$, $C>0$, $\alpha>2$ for $c\in J^*$
and $n\geq 1$, then  $(f^k,\mu)$ satisfies the Central Limit
theorem. If $l_{\max}>1$ for all critical points in $J^*$, then
the density of $\mu$ is $L^p$ for all $1\leq
p<\frac{\l_{\max}}{l_{\max}-1}$.
\end{theorem}

\subsection{Comments on results}

The story about the existence of acip of interval maps has a long
history, see \cite{lu} and reference therein for a comprehensive
survey. Quit general conditions are known to guarantee the
existence of acip for uniformly expanding maps, for smooth maps
with critical points (for S-unimodal maps satisfying
Collect-Eckmann condition in \cite{ce} and Nowicki-van Strien
condition in \cite{ns}, and for multimodal maps under the most
general condition in this setting \cite{brss} recently), for
interval maps with critical points and singularities under
summability conditions in \cite{abv} recently, for smooth maps
with a countable number of critical point in \cite{ap}. Notice
that the results on the smooth case use weaker assumptions than
the case of interval maps with critical points and
discontinuities, an observation is that we only consider the
derivatives of the critical value on the smooth case
\cite{BSS,brss}, while we should need the assumption that the
recurrence of the critical point is not increasing rapidly for
interval maps with critical points and discontinuities.
On the other hand, the results on the existence of acip in smooth
case need an assumption  that the critical orders of each critical
point should equal from two sides in general.

The results regarding decay of correlations and the Central Limit
Theorem
 for the unimodal maps with same critical orders were proved in \cite{yo92,kn}. In \cite{yo92},  Young considered quadratic maps
 satisfying Collect-Eckmann condition and critical recurrence at a sufficiently slow exponential rate, and
 proved that such maps have exponential decay of correlations and satisfy the Central Limit Theorem. Independently,
 Keller and Nowicki \cite{kn} obtained the same results for S-unimodal maps (with same critical orders) satisfying only
 the Collet-Eckmann condition. Later, in \cite{BLS}, Bruin et.al considered multimodal maps (with the same critical orders
 for all critical points and $l(c)>1, \forall c \in C$, note not only from both sides), obtained the same results to
Theorem 2
 under the summability condition $\sum_{n=0}^{\infty}{|Df^n(f(c))|}^{-\frac{1}{2l(c)-1}}< \infty$. Moreover, the construction in that
 paper made it possible to show a direct link between the rate of decay of correlations and the rate of growth of $|Df^n(f(c))|$.
Cedervall \cite{ce2} considered the interval maps with critical
points ( $l(c)>2$, critical
 orders may not equal, but need the critical orders of each critical
point are equal from two sides) under the assumption the
Collet-Eckmann condition and subexponential recurrence conditions,
and proved that such maps have finite number of acips, and
exponential decay of correlations and the Central Limit Theorem.
For the contracting Lorenz map $f$ satisfying
$|Df^n(f(c\pm))|>{\lambda}^n$, for each $n\geq 1$ and some
$\lambda>1$, and $|f^{n-1}(f(c\pm))-c|>{e}^{-\alpha
  n}$ for some $\alpha$ small enough, and for all $n\geq 1$, it
was shown that $f$ admits an acip which has exponential decay of
correlations in \cite{me}.

For interval maps with critical points and singularities, it was
obtained exponential decay of correlations and Central Limit
Theorem,  but need more additional assumptions in
\cite{dhl,da,ly}. Observe that the decay of correlations and
Central Limit Theorem hold for non-H\"{o}lder observables in
\cite{da,ly}.

Note that in the unimodal cases (continuous) with same critical
orders from two sides, $|Df^n(f(c))|\geq Ce^{\lambda n}, C>0,
\lambda>0$ if and only if there is a renormalizaiton $f^n$ with
exponential decay of correlations \cite{nsa}. We are not sure
whether it holds for unidmodal maps with different orders from two
sides, but according Theorem 2, we can get that if the unique
critical point satisfies $\gamma_n(c)\leq Ce^{-\beta n}$ for $C>0$
and $\beta>0$, then there is a renormalizaiton $f^n$ with
exponential decay of correlations.

If $0<l(c) < 1$, $c$ is called a singular point. Since the
negative Schwarzian derivative condition rules out the existence
of singularities, once one can get rid of the negative Schwarzian
derivative condition,  the results in this paper may easily
generalized to interval maps with critical points and
singularities.

\section{Ideas and organization of the proof}
To obtain decay of correlations and the Central Limit Theorem of
interval maps, a useful technique is based on Frobenius-Perron
operator or transfer operator. Exponential decay corresponds to a
gap in the spectrum of this operator, various technique has been
developed for proving the existence of this gap. For example, in
\cite{li}, it was shown that the Frobenius-Perron operator is
contracting with respect to Hilbert metric on defined cones of
density functions for an expanding map. Another powerful tool was
proposed by Young in \cite{yo98}\cite{yo99}. She has shown that
for an induced Markov map, the up-bound of the decay of
correlations are strongly related to the tail estimates of the
inducing time. Notice the method in \cite{yo98}\cite{yo99} could
capture mixing rates that are slower than exponential rates.

Our strategy is to apply the results of Young in \cite{yo99}, so
it is crucial for us to construct an induced Markov map based on
the general maps we considered. Induced Markov map has been
constructed before for unimodal maps with equal critical orders
and satisfying Collet-Eckmann condition, for multimodal maps with
equal critical orders (not only from two sides) and the summabilty
condition $\sum_{n}|Df^n(f(c))|^{-1/(2l(c)-1)}<\infty$ for any
$c\in \mathcal{C} $ \cite{BLS}, for multimodal maps with critical
points ( $l(c)>2$ critical orders may not equal, but need the
critical orders of each critical point from two sides are equal)
under the assumption the Collet-Eckmann condition and
subexponential recurrence conditions \cite{ce2}. Observed that all
the critical orders are equal (from both sides) is necessary in
the construction in \cite{BLS}\cite{ce2}.

In order to carry out almost the same construction \cite{BLS} of
an induced Markov map with corresponding estimates when the
critical points are allowed to have different orders from two
sides, we need do some modifications. Firstly, we need a new
definition of binding period, which involves the recurrence of the
critical points, see the details of the proof on Lemma
\ref{expanding} in Section 3. Secondly, we shall prove the
nonexistence of wandering intervals and backward bound contraction
property (BBC) for maps in $\mathcal{A}$, but we have shown these
results under a weaker condition in \cite{cd} for interval maps
without singularities, recently.

The structure of the paper is as follows. In Section 3, for $f\in
\mathcal{A}_1$, we identify the topological attractor and metric
attractor, and consider the intervals in the attractor. We use a
binding argument to obtain some estimations of growth in terms of
the derivative and recurrence along the appropriate critical
orbits. Similarly, we derived the same estimates for $f\in
\mathcal{A}_2$ on $J^*$. In Section 4, for $f\in \mathcal{A}_1$,
we choose any interval $J$ in the minimal cycle, and construct an
induced Markov map on $J$ which has uniformly distortion and of
the image bounded below on each element of the partition of $J$.
Then we give the induced time estimates. In Section 5, we
construct a full Markov map for an appropriate interval based on
the induced map constructed in Section 4. We also state some
estimates about this full Markov map. In Section 6, we apply
Young's result and present the proof of Theorem 2 and Theorem 3.
For readers who are familiar to the construction in \cite{BLS},
they can skip Sections 4,5, except the proof of Lemma
\ref{defintion reason}.

\section{Notations and some estimates}
For $f\in \mathcal{A}_1$, we will identify the attractors of $f$
which  be both topological and metric in this Section, then we
shall restrict $f$ to the attractors. For $f\in \mathcal{A}_2$, we
restrict $f$ to the interval $J^*$ directly. Taking a small
interval in the each attractor (or $J^*$), we use a binding
argument to obtain estimates in terms of the derivatives and
recurrence along the appropriate critical orbit. Similar argument
have been applied before by Jakobson \cite{ja} and Benedicks and
Carleson \cite{bc} under the strong conditions on $Df^n(f(c))$ and
on the recurrence along the critical orbit. The way we defined
here is imitated on \cite{BLS}, but \cite{BLS} only consider the
growth of $Df^n(f(c))$ in the definition of the binding period,
without the assumption on the rate of recurrence. However, our
definition of binding period relates the recurrence of $f^n(c)$,
then it is useful to tackle the case that the critical points have
different orders (including different orders from both sides).
This is the main point of this work.

A point is called a {\it period point} if $f^n(x)=x$ for integer
$n>0$, it is attracting if its basin include a interior. A general
notation of period orbit is the cycle of interval, if $J\subset M$
is a nontrivial closed interval for which there exists a positive
integer $n$ such that $f^n(J)\subset J$ and $n$ is the least such
integer, we call the set $\cup_{i=0}^{n}f^i(J)$ a {\it cycle} of
intervals for $f$ with period $n$. If the interiors of intervals
in the cycle are pairwise disjoint we say that the cycle is {\it
proper}. If a cycle contains no small cycle, we say it is {\it
minimal}. If $f$ has a proper cycle  $\cup_{i=0}^{n}f^i(J)$,
define $g:M \rightarrow M $ by $g=\Lambda^{-1}\circ f^n \circ
\Lambda$, where $\Lambda$ is an affine transformation  from $M$
onto $J$. We say $g$ is a {\it renormalization} of $f$.

The following Lemma collects some basic properties of proper cycle
and minimal cycle for continuous maps on $M$.
\begin{lemma}\label{cycle}
Let $f$ be a continuous map on $M$, then we have the following
statements: \item (1) the minimal cycle is a proper cycle, \item
(2) minimal cycles ether coincide or have disjoint interiors.
\end{lemma}
\begin{proof}
These follow from the definition of the proper cycle and the
minimal cycle immediately.
\end{proof}

\begin{prop}
If $f$ is a map in $\mathcal{A}_1$, then $f$ has finite
renormalizations, and $f$ has at least one and at most $\sharp
\mathcal{C}$ minimal cycles. Moreover, $\cup_{i=0}f^{-i}(c)$ for
each $c$ in the minimal cycle is dense in the corresponding
minimal cycle.
\end{prop}
\begin{proof}
At first, since $f$ satisfies some summability conditions and
negative Schwarzian derivative, $f$ has no attracting or neutral
periodic orbits by Singer's Theorem \cite{SI}. On the other hand,
$f$ has no wandering interval by \cite{cd}, thus from the
contraction principle, there exists $\delta>0$ such that  for any
interval $J\in M$ with $|J|>0$,  there exists $N_J>0$ such that we
have $f^n(J)>\delta$ for $n>N_J$, Hence all intervals constituting
a cycle have the length greater than $\delta$, this implies that
the period of any proper cycle is bound by ${\delta}^{-1}$. Thus
$f$ has a finite number of renormalizations.

Next, we suppose that all intervals in a cycle
$\cup_{i=0}^{n}f^i(J)$ don't contain a turning point of $f$ in its
interior, then $f^n$ or $f^{2n}$ is monotone increasing and
continuous on $J$, and $f^n(J)\subset J$ or $f^{2n}(J)\subset J$,
this contradicts $f$ has no attracting or neutral period orbits,
then $f$ has at most $\sharp \mathcal{C}$ minimal cycles. On the
other hand, if $M$ is not a minimal cycle of period one, by the
finiteness of proper cycles, then there is a cycle could contain
no small cycle, i.e., this is a minimal cycle.

Finally, we consider a minimal cycle $\cup_{i=0}^{n}f^i(J)$, $c$
is a critical point in $\cup_{i=0}^{n}f^i(J)$. Because $f$ has no
wandering interval, and has no attracting and neutral periodic
orbits, $f$ has no homterval. Therefore, any small interval in the
cycle will eventually visit the critical point in the cycle. If
the preimages of $c$ isn't dense in the minimal cycle, we can find
a cycle smaller the minimal cycle, this would induce a
contradiction.

\end{proof}

For $f\in \mathcal{A}_1$ and each minimal cycle
$X:=\cup_{i=0}^{m}f^i(J)$ with period $m$, denote
$\mathcal{C}_c\in \mathcal{C}$ as critical set in $X$, we will
consider the subsystem $(X, f^m)$. For $f\in \mathcal{A}_2$, we
consider the subsystem $(J^*, f)$ directly, denote
$\mathcal{C}_c\in \mathcal{C}$ as critical set in $J^*$ too.

For $x\in X$ or $x\in J^*$, let $c$ be the critical point closest
to $x$, Given a critical neighborhood $\bigtriangleup$ of
$\mathcal{C}_c$, we define the {\it binding period} as follows,

\begin{equation}
p(x):=
\begin{cases} \max \Big{\{}p; |f^k(x)-f^k(c)|\leq \gamma_k(c)|f^k(c)-\mathcal{C}|,  \ 1\leq k \leq p-1 \Big{\}}  \ \ & \  x \in
\bigtriangleup,
\\
0  \ \ & \  x\notin \bigtriangleup.
   \end{cases}
   \end{equation}

The size of the critical neighborhood $\bigtriangleup$ depends on
the following Lemmas.

\begin{lemma}(BBC property)\label{bbc}
Let $f\in \mathcal{A}$, then there exists $K>0$ such that for all
$\delta_0>0$ there exist $0<\delta<\delta_0$,
$\bigtriangleup_{\delta}=\cup_{c\in \mathcal{C}}(c-\delta,
c+\delta)$, for each $x\in M$ we have
\begin{equation}
|Df^n(x)|>K, \ \ \ where \  n=\min\{i\geq 0; f^i(x)\in
\bigtriangleup_{\delta}\}.
\end{equation}
\end{lemma}
\begin{proof}
See Theorem C in \cite{cd}.
\end{proof}

{\remark In fact, BBC holds for maps satisfying a weaker condition
than $\mathcal{A}$ \cite{cd}. For symmetric unimodal maps with
negative Schwarzian derivative, BBC holds \cite{MM}. For the
multimodal case which be with the same critical orders of all
critical points and an increasing condition
$$\lim_{n\to \infty} |Df^n(f(c))|=\infty, \ \ \forall c\in \mathcal{C},$$ it was shown in
\cite{BS03}. }

\begin{lemma}\label{Uniformly expanding}(Uniformly expanding outside of $\bigtriangleup_{\delta}$) Let $f\in \mathcal{A}$, there
exist $C({\delta})>0$ and $\lambda_{\delta}>0$ such that for
$x,f(x), \dots, f^{n-1}(x)\notin \bigtriangleup_{\delta}$, then
\begin{equation}
|Df^n(x)|\geq C({\delta}) e^{\lambda_{\delta}n}.
\end{equation}
\end{lemma}

\begin{proof}
 Since $f\in \mathcal{A}$, all the periodic orbits of $f$ are
 repelling \cite{SI}. We can define a new map $\tilde{f}$ such that $\tilde{f}$
 is $C^2$ in $\bigtriangleup_{\delta_1}$ and has no change with  $f$ outside of
 $\bigtriangleup_{\delta_1}$, then the above is a consequence of Ma{\~n}{\'e}
 theorem \cite{MA}.
\end{proof}

\begin{lemma}\label{small}
Suppose $G_p\geq 0$ and $\sum_pG_p< \infty$, then for any
$\zeta>0$ there exists $p_0>0$ such that
\begin{equation}
\sum_{s\geq 1}\sum_{(p_1,\dots,p_s) and p_i\geq
p_0}\prod_{p_i}\zeta G_{p_i}\leq 1.
\end{equation}
\end{lemma}
\begin{proof}
See  \cite{BLS}.
\end{proof}

For each $c\in \mathcal{C}_c$ or $c\in J^*$, let
$\bigtriangleup:=\cup_{c\in \mathcal{C}_c \ or \ J^*}(c-\delta, \
c+\delta),\ p_{\delta}:=p(c+\delta)\ or \ p(c-\delta)$ depending
the neighborhood we consider. Note that $p_{\delta}\rightarrow
\infty$ as $\delta \rightarrow 0$. Using Lemmas [\ref{bbc},
\ref{Uniformly expanding}, \ref{small}], and the summability of
$\gamma_n(c)$ for each $c$, we can fix at this moment and for the
rest of the paper $\delta$ so small that

(1) BBC holds and uniformly expanding outside of $\bigtriangleup$,

(2)$\sum_{s\geq 1}\sum_{(p_1,\dots,p_s)\  and\ p_i\geq
p_{\delta}}\prod_{p_i}\zeta \gamma_{p_i}(c)\leq 1$ for all $c\in
\mathcal{C}_c \ or\  J^{*}$, where $\zeta$ is a constant which
depends only the map itself and will be specified explicitly in
the proof, see the proof of Lemma \ref{defintion reason}.

The following Lemma gives an estimation of derivative growth for
points in $\bigtriangleup$.

\begin{lemma}\label{expanding}
Let $f\in \mathcal{A}$, denote $I_p:=\{x; \ \ p(x)=p\}$,  and for
each critical point $c\in \mathcal{C}_c$ (or $J^*$), put
$DF_p(c):=\min\{|Df^p(x)|; x\in I_p \cap \bigtriangleup \}$, then
there exists constant $K_1>0$ such that
\begin{equation}
DF_p(c)\geq K_1\gamma^{-1}_p(c).
\end{equation}
\end{lemma}

\begin{proof}
For any interval $B\in M$ and integer $n \geq 1$, let $B_j=f^j(B)$
for $j=0,... ,n$ and define the {\it generalized distortion}:
$$
D(f^n, B)=\prod_{j=0}^{n-1}\sup_{x_j, y_j \in
B_j}\frac{|Df(x_j)|}{|Df(y_j)|}.
$$
This definition and the mean value theorem imply
$$\frac{Df(x_j)}{Df(y_j)}=1+\frac{Df(x_j)-Df(y_j)}{Df(y_j)}=1+\frac{Df^2(\xi_j)}{Df(y_j)}(x_j-y_j)$$
where $\xi_j\in (x_j, y_j)$, and
\begin{equation}\label{generalized distoration}
D(f^n, B)\leq \prod_{j=0}^{n-1}(1+ \frac{\sup_{B_j}
|D^2f|}{\inf_{B_j}|Df|} |B_j|).
\end{equation}

We start by considering the points in $\bigtriangleup$, for any
$x\in I_p \cap \bigtriangleup$, $c$ is the closest critical point
to $x$.

{\bf Claim:} There exists a positive constant $K$ independent of
$x$ such that for any $1\leq k\leq p(x)-1$, then
$$D(f^k, [f(x),f(c)]) \leq  K.$$

Indeed,  put $B_0=[x,c]$ and $B_j=f^j(B_0)$, and denote
$d({B_j})=dist(B_j, \mathcal{C})$, from the definition of binding
period, and $\gamma_j(c)<\frac{1}{2}$ for every $1\leq j\leq
p(x)-1$, we have
 \begin{equation}\label{in the neighborhood}
 \begin{split}
 \frac{|{B_j}|}{d({B_j})}
  &\leq
\frac{|f^j(x)-f^j(c)|}{|f^j(c)-\mathcal{C}|-|f^j(c)-f^j(x)|} \\
&\leq
\frac{\gamma_j(c)|f^j(c)-\mathcal{C}|}{(1-\gamma_j(c))|f^j(c)-\mathcal{C}|}=\frac{\gamma_j(c)}{1-r_j(c)}\leq
2\gamma_j(c).
\end{split}
\end{equation}
By the orders of the critical points, we obtain that
$$
\sup_{x_j, y_j \in {B}_j}
\frac{|D^2f(x_j)|}{|Df(y_j)|}=\frac{\sup_{{B}_j}|D^2f|}{\inf_{{B}_j}|Df|}
\leq \frac{{K_l}^2}{d{({B}_j)}},
$$
where $K_l$ is a constant from the orders of the critical point.
Combining (\ref{generalized distoration}) and (\ref{in the
neighborhood}) we have
 \begin{equation*}
 \begin{split}
D(f^k, {B}_1)&\leq \prod_{j=1}^{k-1}(1+\sup_{x_j, y_j \in
{B}_j} \frac{|D^2f(x_j)|}{|Df(y_j)|}|B_j|)  \\
&\leq \prod_{j=1}^{k-1}(1+ {K_l}^2\frac{|{B}_j|}{d({B}_j)}) \leq
\prod_{j=1}^{k}(1+2{K_l}^2\gamma_j(c)).
\end{split}
\end{equation*}
Using the inequality $\ln(1+x)\leq x$, and the  summability of
$\gamma_j(c)$, we obtain  the uniform bound of the general
distortion $D(f^k, {B}_1).$ The Claim follows.

Using the above Claim, we have
\begin{equation*}
\begin{split}
|Df^p(x)|&= |Df^{p-1}(f(x))||Df(x)| \geq \frac{K}{K_l}
|Df^{p-1}(f(c))|{|x-c|}^{l(c)-1} \\
&\geq K {K_l}^{\frac{1}{l(c)}}
|Df^{p-1}(f(c))|{|f(x)-f(c)|}^{(l(c)-1)/l(c)}.
\end{split}
\end{equation*}
On the other hand, by the mean value theorem and the definition of
the binding period, we obtain
$$
K |Df^{p-1}(f(c))||f(x)-f(c)|\geq |f^p(x)-f^p(c)|\geq
\gamma_p(c)|f^p(c)-\mathcal{C}|.
$$
So it concludes
 that $|f(x)-f(c)|\geq
\frac{\gamma_p(c)|f^p(c)-\mathcal{C}|}{ K|Df^{p-1}(f(c))|}$.
Therefore,
$$|Df^p(x)|\geq K {K_l}^{\frac{1}{l(c)}}
{(\frac{\gamma_p(c)|f^p(c)-\mathcal{C}|}{
K|Df^{p-1}(f(c))|})}^{(l(c)-1)/l(c)}|Df^{p-1}(f(c))|.
$$
From the orders of the critical points, we can assume that
$|Df(f^p(c))|\leq K_l{|f^p(c)-\tilde{c}|}^{l(\tilde{c})-1}$, where
$\tilde{c}$ is the closest critical point to $f^n(c)$. Put
$K_1={KK_l}^{1/l(c)}$, we have

\begin{equation}
\begin{split}
|Df^p(x)|&\geq K_1
{\gamma_p(c)}^{\frac{l(c)-1}{l(c)}}{|Df^{p-1}(f(c))|}^{\frac{1}{l(c)}}
{|f^p(c)-\tilde{c}|}^{\frac{l(\tilde{c})-1}{l(c)}}
{|f^p(c)-\tilde{c}|}^{\frac{l(c)-l(\tilde{c})}{l(c)}}\\
&\geq K_1
{\gamma_p(c)}^{\frac{l(c)-1}{l(c)}}{|Df^{p}(f(c))|}^{\frac{1}{l(c)}}
{|f^p(c)-\tilde{c}|}^{\frac{l(c)-l(\tilde{c})}{l(c)}}=K_1
\gamma^{-1}_p(c).
\end{split}
\end{equation}
The proof of Lemma is complete.
\end{proof}

\section{The construction of the induce map}
Our aim in this Section is to construct a countable partition
$\hat{P}$ of an interval $J$ into open intervals, define an
inducing time function $\tau: \hat{P} \rightarrow N$ which is
constant on elements of $\hat{P}$, and let $\hat{F}
:\hat{P}\rightarrow M$, denote the induce map by
$$
\hat{F}(x)=f^{\tau(x)}(x).
$$
We will show this induce map has uniformly bounded distortion and
its image has bounded below on each element of $\hat{P}$. We will
give the corresponding estimates about the induce time in the last
Subsection. This construction is essentially indented to that of
\cite{BLS}, we will use the estimates without proof if these
estimates will not effect by the difference of the critical orders
from two sides. More precisely, it is the following,

\begin{prop}
Let $f\in \mathcal{A}_1$, then there exists $\delta'>0$ such that
for all $\delta''>0$ the following properties hold. For an
arbitrary interval $J\subset X$ with  $|J|\geq \delta''$, there
exists a countable partition $\hat{P}$ of $J$ and an induced time
function $\hat{p}:\hat{P}\rightarrow N$ constant on each element
$w$ of $\hat{P}$, such that the induced map
$\hat{F}=f^{\hat{p}(x)}(x)$ has  uniformly bounded distortion and
$|\hat{F}(w)|=|f^{\hat{p}(w)}(w)| \geq \delta' $. Moreover, it
satisfies the following estimates.

(1)(Summability induced time)
$$
\sum_n|\{\hat{p}>n|J\}|<\infty.
$$

(2)(Polynomial inducing time)
 If $d_n(c)<Cn^{-\alpha}$, $C>0$, $\alpha>1$ for all $c\in \mathcal{C}_c$ and $n\geq 1$, then
 there exists $\hat{C}>0$ such that
 $$
 |\{\hat{p}>n|J\}|<\hat{C}n^{-\alpha}.
 $$

(3)(Stretched exponential case) If
$\gamma_n(c)<Ce^{-\beta{n}^{\alpha}}$, $C>0$, $\alpha\in (0,1)$,
$\beta>0$ for all $c\in \mathcal{C}_c$  and $n\geq 1$, then for
each $\hat{\alpha}\in (0, \alpha)$, there exist $\hat{\beta},
\hat{C}>0$ such that
$$
|\{\hat{p}>n|J\}|<\hat{C}e^{-\hat{\beta}{n}^{\hat{\alpha}}}.
$$

(4)(Exponential case) If $\gamma_n(c)<Ce^{-\beta{n}}$, $C>0$,
$\beta>0$ for all $c\in \mathcal{C}_c$ and $n\geq 1$, then there
exist $\hat{\beta}, \hat{C}>0$ such that
$$
|\{\hat{p}>n|J\}|<\hat{C}e^{-\hat{\beta}{n}}.
$$
\end{prop}

At first, we clarify the role of the constant and some notations
in the above Proposition. Let $\bigtriangleup_1=\cup_{c\in
\mathcal{C}_c}(c-\frac{\delta}{2}, c+\frac{\delta}{2})$, using a
result in \cite{cd}, i.e., for any Borel set $A$, there exists
$C>0$ such that $|f^{-n}(A)| \leq C {|A|}^{\frac{1}{l_{\max}}}$,
then there exists $\delta'>0$ such that for each component $w\in
\bigtriangleup_1\backslash \mathcal{C}_c$ and each $n\geq 0$,
$|f^n(w)|\geq \delta'$ \footnote{In fact, it suffices to choose
$\delta'$ so that $f^i(w)\geq \delta'$ for each $w\in
\bigtriangleup_1\backslash \mathcal{C}_c$ and $0\leq i \leq
n=\min\{k; Int(f^i(w))\cap \mathcal{C} \neq \emptyset\}$, where
$Int(f^i(w)$ is the interior point of $f^i(w)$.}. We also suppose
$\delta'\leq \frac{\delta}{2}$, where $\delta$ is a constant we
have fixed in previous Section. $\delta''$ is a constant to be
fixed in next section. We denote $|\{\hat{p}>n|J\}|$ by the
conditional probability
$$
\frac{|\{x\in J; \hat{p}(x)>n\}|}{|J|}.
$$

\subsection{The construction of the induced map $\hat{F}$}

Denote $\bigtriangleup_1=\cup_{c\in C_c}(c-\frac{\delta}{2},
c+\frac{\delta}{2})$ as above. For any interval $J$ with $|J|\geq
\delta''$, we subdivide $J$ by the $I_p, p\geq 0$. For each
subinterval $w=I_p\cap J$, let
$$
\nu_1=\min\{n\geq 0; \ f^n(w)\cap \bigtriangleup_1\neq
\emptyset\},
$$
be the first visit of $w$ to $\bigtriangleup_1$, denote
$\tilde{w}=f^{\nu_1}(w)$, we distinguish two cases:

(1). $|\tilde{w}|<\delta'$,  we subdivide $\tilde{w}$ with the
elements $\{I_p\}$. Each interval $I_p\cap \tilde{w}$ for $p>0$ is
labeled as deep return. Notice that by the choice of $\delta'$,
each return in this case is a deep return.

(2). $|\tilde{w}|\geq \delta'$,  we cut off two side intervals of
length $\epsilon$ from $\tilde{w}$, where $\epsilon$ is a small
parameter to be fixed (see Lemma \ref{defintion reason}). The
middle part is called the large scale, add the sub interval
$w_0\subset w$ to the partition that $f^{\nu_1}(w_0)$ is equals to
this middle part of $\tilde{w}$, and stop to work on this middle
part. Suppose that $\tilde{w}\pm$ are the two pieces of length
$\epsilon$ that are cut off, we subdivide $\tilde{w}\pm$ further
by the elements in $\{I_p\}$. Each interval $I_p\cap \tilde{w}\pm$
for $p>0$ is labeled as deep return, $I_0\cap \tilde{w}\pm$ is
labeled as shallow return.

Let $w'$ be a partition interval by $\{I_p\}$ in $\tilde{w}$ which
has not reached the large scale, denote the binding period by
$p(w')$, let
$$
\nu_2 = \nu_1 +\min\{n\geq p(w');\  f^n(w')\cap
\bigtriangleup_1\neq \emptyset \}.
$$
Then we subdivide  $f^{\nu_2-\nu_1}(w')$ according above rules,
stop until some parts reach the large scale. Notice that we have
applied binding period (if exist) time iteration to guarantee
expansion after each return time if it has not reach large scale.

We then construct inductively a sequence of partitions $\hat{P}_n$
by only considering at most $n$ iterates, denote $\hat{P}$ by the
partition of $J$ by considering all iterates of $f$.

Let $x\in w$ where the above algorithm eventually stops, then
there exists $n>0$ such that $x\in \hat{P}_n$, denote the stopping
time by $\hat{P}_J(x)=n$, otherwise, set  $\hat{P}_J(x)=\infty$.

Let $\hat{J}= \{x\in J; \hat{P}_J(x) < \infty \}$, we have defined
the induced map
$$
\hat{F}_J: \hat{J}\rightarrow \ M , \
\hat{F}_J(x)=f^{\hat{P}_J(x)}(x).
$$

In next subsection we will show $\hat{J}$ is a partition of $J$ up
to a set of zero Lebesgue measure, and give the corresponding
estimates of $\hat{F}_J$, but now we first want to know the
structure of $\hat{J}$. Given an interval $w$ in some partition
$\hat{P}_n, n>0$, we associated it with  a sequence
$$
(\nu_1, p_1),\ (\nu_2, p_2), \dots,\ (\nu_s, p_s),
$$
where $\nu_i,1\leq i \leq s$ is the return times to
$\bigtriangleup_1$, $p_i$ is the corresponding periods, and $s$ is
the maximum integer  such that $\nu_s\leq n$. If $w$ is an
interval on which $(\nu_1, p_1),\ (\nu_2, p_2), \dots,\
(\nu_{j-1}, p_{j-1}),$ is fixed and $\nu_j$ is the next return,
then $\{x\in w; p(f^{\nu_j}x)=p\}$ has at most $4$ components.
This maximum is attained when $|f^{\nu_j}(w)|\geq \delta$, and the
outmost intervals of length of $\epsilon$ contain a critical point
and $p_j$ is big enough.

Notice that the corresponding sequence $(\nu_1, p_1),\ (\nu_2,
p_2), \dots,\ (\nu_s, p_s)$ contains information  about the
iterations of $f$ on $w$, although many sequences  don't
correspond to partition intervals. For a given sequence  $(\nu_1,
p_1),\ (\nu_2, p_2), \dots,\ (\nu_s, p_s)$, let
$$
S_d:=\{i\leq s; \nu_i\  is \ a \ deep \ return\}=\{i\leq s; \
p_i>0\},
$$
$$
S_s:=\{i\leq s; \nu_i\  is \ a \ shallow\ return\}=\{i\leq s; \
p_i=0\},
$$
$$
S_{s,s}:=\{i<s; \ p_i=0\  and \ p_{i+1}=0\}.
$$

\subsection{Important estimates}

\begin{lemma}\label{bound distortion 1}
There exists $K_{\epsilon}>0$ such that for all $w\in \hat{P}$,
the distortion of $\hat{F}_J|_w$ is bounded by $K_{\epsilon}$.
\end{lemma}

\begin{proof}
Let $w\in \hat{P}_n$, then $f^n(w)$ has reached large scale. By
the construction of the induced map, there is an interval
$T\supset w$ such that $f^n(T)$ is a $\epsilon$-scaled
neighborhood of $f^n(w)$, i.e.,$|f^n(T)\setminus f^n(w)|\geq
\epsilon |f^n(w)|$. On the other hand, it is easy to see that
$f^n$ is a diffeomorphism on $T$. Using the Koebe principle, we
can obtain the result.
\end{proof}

For a given sequence $(\nu_1, p_1),\ (\nu_2, p_2), \dots,\ (\nu_s,
p_s)$, the following Lemma contains an important metric estimation
of the length of the corresponding interval.

\begin{lemma}\label{important}
Let $C=C({\frac{\delta}{2}})$ and $\lambda=\lambda_{\delta/2}$ be
as in Lemma \ref{Uniformly expanding}, and $K$ be constant in
Lemma \ref{bbc}. There exist $K_0>0$ independent of $\epsilon$ and
$\rho\in (0,1)$ with the following properties. For a given
sequence $(\nu_1, p_1),\ (\nu_2, p_2) \dots,\ (\nu_s, p_s)$ with
$\nu_s\leq n$, the corresponding interval $w_{p_1p_2,\dots,p_s}\in
\hat{P}_n$, we have
$$
\frac{|w_{p_1p_2,\dots,p_s}|}{|f^m(w_{p_1p_2,\dots,p_s})|} \leq \
\min \Big{\{} C^{-\sharp S_d}e^{-\lambda (m-\sum_{i=0}^{s}p_i)}, \
{(\frac{K_0}{K})}^{\sharp S_d} {\rho}^{\sharp S_{s,s}} \Big{\}}
\prod_{i\in S_d}{(DF_{p_i})}^{-1}
$$
for $m=\max\{n, \nu_s+p_s\}$. And $\rho\to 0$ as $\epsilon \to 0$.
Moreover there exists $T>0$ which can be chosen arbitrarily large
if $\epsilon$ is small, such that $\nu_{i+1}-\nu_{i}\geq T$
whenever $p_i=p_{i+1}=0$.
\end{lemma}
\begin{proof}
See the proof of Lemma 3.2 in \cite{BLS}.
\end{proof}

\subsection{Induced time estimates}

The aim of this subsection is to estimate the tail behaviors of
the induced time function $\hat{P}$, i.e., the estimation of
$\{x\in J; \hat{p}(x)>n \}$.

We fixed $n$ for the rest of this subsection. Let $\eta>0$ be a
small constant to be determined in Lemma \ref{time estimate}, we
can divide the elements in $\hat{P}_n$ with $\hat{p}|_w>n$ into
two parts
$$
\hat{P}_n'=\{w\in \hat{P}_n; \hat{p}|_w>n, \sum_{s=1}^sp_i\leq
\eta n \},
$$
$$
\hat{P}_n''=\{w\in \hat{P}_n; \hat{p}|_w>n, \sum_{s=1}^sp_i> \eta
n \}.
$$
Then we have
$$
|\{\hat{p}>n\}|=\sum_{w\in \hat{P}_n'}|w| + \sum_{w\in
\hat{P}_n''}|w|.
$$

To treat the exponential and stretched exponential cases, we
subdivide $\hat{P}_n''$ further into
$$
\hat{P}_{n-}''=\{w\in \hat{P}_n'', \ s\leq \rho
n^{\hat{\alpha}}\}, \ \ \hat{P}_{n+}''=\{w\in \hat{P}_n'', \ s>
\rho n^{\hat{\alpha}}\},
$$
where $\hat{\alpha}\in (0,1]$ and $\rho>0$ are constants to be
fixed below.

Intuitively, ${p_1, p_2,\dots,p_s}$ in $\hat{P}_n'$ is small, then
the elements in $\hat{P}_n'$ spend much time in $X \setminus
\bigtriangleup$. Thus we can use Lemma \ref{Uniformly expanding}
to get exponential rate of decay of $\hat{P}_n'$. For elements in
$\hat{P}_n''$, they spend much time in $\bigtriangleup$, then the
estimations of $\hat{P}_n''$ relate closely to the property of the
critical point. So we obtain the estimations of $\hat{P}_n''$ by
using Lemma \ref{bbc} and Lemma \ref{important} under the
assumption of the rates of decay of $\gamma_n(c)$.

\begin{lemma}\label{time estimate}
For any $\theta>0$ there exists $\eta_0>0$ such that for all
$0<\eta<\eta_0$ and for $n$ sufficiently large,
$$
\sum_{w\in \hat{P}_n'}|w| \leq e^{-(\lambda-\theta)n}.
$$
\end{lemma}
\begin{proof}
See the proof of Lemma 3.5 in \cite{BLS}.
\end{proof}

\begin{lemma}\label{defintion reason}
Fix $L\in \{1,\dots, n\}$ arbitrary and let
$$
\hat{d}_{n,s}(c) := d_i(c) \ for \ i= \max\{[\frac{\eta n}{2s^2}],
L\}.
$$
Write $s(w)=s$ if the associate sequence $(\nu_1, p_1),\ (\nu_2,
p_2), \dots,\ (\nu_s, p_s)$ of $w$ has length $s$. For any
$\eta>0$ there exists $C_1>0$ such that
$$
\sum_{w\in \hat{P}_n'', s(w)\geq L}|w|\leq C_1 \max_{c\in
\mathcal{C}_c} \sum_{s=L}^n 2^{-s} \hat{d}_{n,s}(c).
$$
\end{lemma}
\begin{proof}
(It suffices to take our definition of $\gamma_n(c)$ and use the
proof of Lemma 3.6 in \cite{BLS}.) Given a sequence $(\nu_1,
p_1),\ (\nu_2, p_2), \dots,\ (\nu_s, p_s)$, let $p_{j'}$ be the
first return term such that $p_{j'}\geq \frac{\eta n}{2{j'}^2}$.
Because $p_1+p_2+\dots+ p_s\geq \eta n$, such $j'$ exists. Take
$j=\max\{L,j'\}$.

Let $\tilde{w}_{p_1p_2,...,p_s}$ be the union of adjacent
intervals $w_{p_1p_2...p_{j-1}p}$ with common return times
$\nu_1,\dots,\nu_j$ and $p\geq p_j$. Then $f^{\nu_j}$ maps
$\tilde{w}_{p_1p_2,\dots,p_j}$ diffeomorphically  onto a interval
$(x,y)$ such that $p(x),p(y)\geq p_j$. Assume without loss of
generality that $|x-c|\geq |y-c|$ and $\tilde{c}$ is the closest
critical point to $f^i(c)$, therefore for $i<p_j$,
\begin{equation}
\begin{split}
\gamma_i(c)|f^i(c)-\tilde{c}|&\geq |f^i(x)-f^i(c)| \geq K
|Df^{i-1}(f(c))||f(x)-f(c)| \\
&\geq K |Df^{i-1}(f(c))|{|x-c|}^{l(c)}\\
&\geq KK_l
\frac{|Df^{i}(f(c))|{|x-c|}^{l(c)}}{{|f^i(c)-\tilde{c}|}^{l(\tilde{c})-1}},
\end{split}
\end{equation}
where $K$ is the constant in the proof of Lemma \ref{expanding}.
This gives that
$$
\gamma_i(c){|f^i(c)-\tilde{c}|}^{l(\tilde{c})}\geq KK_l
|Df^{i}(f(c))|{|x-c|}^{l(c)}.
$$
Thus
$$
|x-y|\leq 2|x-c|\leq 2KK_l
{(\frac{\gamma_i(c)}{|Df^{i}(f(c))|})}^{1\diagup l(c)}
{|f^i(c)-\tilde{c}|}^{l(\tilde{c})\diagup l(c)}.
$$
Since the above inequality holds for all $i\leq p_j$, and by the
definition of $d_{p_j}(c)$,  we can obtain that
$$
|x-y|\leq 2KK_l  d_{p_j}(c) \leq 2KK_l \max_{p\geq \frac{\eta
n}{2j^2}}d_p(c)= 2KK_l \hat{d}_{n,j}(c).
$$

Let $S_s$ and $S_d$ be shallow returns times and deep return times
before the indices smaller than $j$, $S_d':=S_d\backslash\{j\}$,
since $f^{\nu_j}$ maps $\tilde{w}_{p_1p_2,\dots,p_j}$
diffeomorphically onto an interval $(x,y)$,  and by Lemma
\ref{important}, we have
\begin{equation}
\begin{split}
\sum_{w\in \hat{P}_n'', s(w)\geq L}|w|&\leq \sum_{j=L}^n
\sum_{p_1,p_2,\dots,p_j}|\tilde{w}_{p_1,\dots,p_j}| \\
&\leq \sum_{j=L}^{n} 2KK_l\max_{c\in
\mathcal{C}_c}\hat{d}_{n,j}(c)
\sum_{p_1,p_2,\dots,p_{j-1}}4^j(\frac{K_1}{K})^{\#S_d}{\rho}^{\#S_{s,s}}\prod_{i\in
S_d'}\frac{1}{DF_{p_i}}.
\end{split}
\end{equation}
On the other hand, by an element calculation,
$$
\sum_{p_1,p_2,\dots,p_{j-1}}4^j\frac{K_1}{K}^{\#S_d}{\rho}^{\#S_{s,s}}\prod_{i\in
S_d'}\frac{1}{DF_{p_i}} \leq 2^{-j}\frac{512K_1}{K}
\sum_{p_1,p_2,\dots,p_{j-1}} {(8\rho)}^{\#S_{s,s}}\prod_{i\in
S_d'}\frac{64K_1}{KDF_{p_i}}.
$$
Notice that $p_i\geq p_{\delta}$ for $i\in S_d$, and we can take
$\epsilon$ so small that $\rho\leq1/8$. Thus by Lemma \ref{small}
and the choose of $\delta$, we have
$$
\sum_{p_1,p_2,\dots,p_j\geq p_{\delta}}
{(8\rho)}^{\#S_{s,s}}\prod_{i\in S_d'}\frac{64K_1}{KDF_{p_i}}\leq
1.
$$
Therefore, take $C_1=1024K_1K_l$, the Lemma follows.
\end{proof}

To distinguish the exponential and stretched exponential case, we
need the following Lemma.
\begin{lemma}\label{exponential}
Assume that there exist $C,\beta>0$ and $\alpha\in (0,1]$ such
that $\gamma_n(c)\leq Ce^{-\beta n^{\alpha}}$ for all positive
integer $n$ and $c\in \mathcal{C}_c$. Then for each
$\hat{\alpha}\in (0,\alpha)$ (or $\hat{\alpha}=1$ if $\alpha=1$)
there exist $\rho,C',\beta'$ such that
$$
\sum_{w\in \hat{P}_{n_-}''}|w|\leq C'e^{-\beta'n^{\alpha}}.
$$
\end{lemma}

\begin{proof}
See the proof of Lemma 3.7 in  \cite{BLS}.
\end{proof}

{\it Proof of Proposition 3.1} We first show that $\hat{J}$ has
full measure in $J$. Because $f$ has no wandering intervals
\cite{cd}, almost all $x\in J$, $f^{n_k}(x)\rightarrow
\mathcal{C}_c$ as $n_k\rightarrow \infty$. Hence $x$ has
infinitely many deep returns (if it has not reach large scale),
assume that $x$ is contained in $w_{p_1,\dots,p_s}$ with arbitrary
$s$. By Lemma \ref{time estimate} and Lemma \ref{defintion
reason}, we have
$$
|\{x\in J; \hat{p}(x)=\infty \}|\leq \lim_{s \rightarrow \infty}
\sum_{(p_1,\dots,p_s)}|w_{(p_1,\dots,p_s)}|=0.
$$
By the construction of the induce map $\hat{F}$, $\hat{F}$ is a
diffeomorphism with uniformly bounded distortion, and
$|\hat{F}(w)|\geq \delta'$ on any component $w$ of $\hat{J}$.

Next we will show the estimates mentioned in the Proposition.
Notice that we have exponential bounds for $\hat{P}_n'$ by Lemma
\ref{time estimate}, so we only consider $\hat{P}_n''$ here. First
observe that for each $k$ there are at most
$$
\#\{n; \ k-1\leq \eta n/2s^2\leq k\}\leq 2s^2/\eta
$$
numbers $n$ such that $k=[\eta n/2s^2]$. Therefore, using Lemma
\ref{defintion reason} with $L=1$,
\begin{equation*}
\begin{split}
\sum_{n\geq 1}\sum_{s=1}^n \hat{d}_{n,s}(c)&\leq \sum_{s\geq 1}
\frac{2s^2}{\eta} 2^{-s} \sum_{k\geq 1}
{(\gamma_k(c)/|Df^k(f(c))|)}^{1/l(c)}{|f^k(c)-\hat{c}|}^{l(\hat{c})/l(c)}\\
&\leq \frac{12}{\eta}\sum_{k\geq 1}\gamma_{k-1}(c) \leq \infty.
\end{split}
\end{equation*}
Hence
$$
\sum_{n}|\hat{p}>n|J|=\sum_n|\hat{P}_n'|+\sum_{n}|\hat{P}_n''|
\leq \infty.
$$

For polynomial case, take $L=1$ in Lemma \ref{defintion reason},
if  $\gamma_n(c)<Cn^{-\alpha}$, $\alpha>1$ for all $c\in
\mathcal{C}_c$,
 and $n\geq 1$, then
\begin{equation*}
\begin{split}
\sum_{w\in \hat{P}_n''}&\leq C_1 \max_{c\in
\mathcal{C}_c}\sum_{s=1}^{n}2^{-s} \hat{d}_{n,s}(c)\leq
C_1\max_{c\in
\mathcal{C}_c}\sum_{s=1}^{n}2^{-s}d_{[\frac{\eta n}{2s^2}]}(c)\\
&\leq C_1\max_{c\in C_c} \sum_{s=1}^{n} 2^{-s}{[\frac{\eta
n}{2s^2}]}^{-\alpha} \leq \frac{12C_1}{{(\eta n)}^\alpha}.
\end{split}
\end{equation*}

For the stretched exponential case, if $\gamma_n(c)\leq Ce^{-\beta
n^{\alpha}}$ for all $n$ and $c\in \mathcal{C}_c$, using Lemma
\ref{defintion reason} and take $L=\rho {n}^{\hat{\alpha}}$, we
obtain
\begin{equation*}
\begin{split}
\sum_{w\in \hat{P}_{n_+}''}&\leq C_1 \max_{c\in
\mathcal{C}_c}\sum_{s= \rho n^{\hat{\alpha}}}^{n}2^{-s}
\hat{d}_{n,s}(c)\leq C_1\max_{c\in \mathcal{C}_c}\sum_{s=
\rho n^{\hat{\alpha}}}^{n}2^{-s}d_{[\frac{\eta n}{2s^2}]}(c)\\
&\leq C_1 2^{-\rho n^{\hat{\alpha}}}\max_{c\in \mathcal{C}_c}
d_{\rho n^{\hat{\alpha}}}(c)\leq C_1 e^{-\rho'n^{\hat{\alpha}}},
\end{split}
\end{equation*}
where the third inequality follows from that $d_n(c)$ is
decreasing as $n$. On the other hand, by Lemma \ref{exponential},
we have for each $\hat{\alpha}\in (0,1)$,
$$
\sum_{w\in \hat{P}_{n_-}''}|w| \leq C'e^{-\beta n^{\alpha}}.
$$
Since $\hat{\alpha}<\alpha$, we obtain result in the stretched
case.

For the exponential case, it suffices to take $\alpha=1$ in the
stretched case.

So we obtain all results claimed in Proposition.

\section{The construction of the full markov map}
In this Section we will construct a full Markov map $\hat{f}$ and
give its estimates based on the map $\hat{F}: \hat{J} \rightarrow
X$ constructed in previous section. Let $J\in X$ be an arbitrary
interval, $\hat{P}$ be the corresponding partition constructed in
Section 4.

\begin{lemma}\label{full Markov}
Let $f\in \mathcal{A}_1$. There exists an open interval $\Omega$
around each critical point $c$ that is contained in a minimal
cycle $X$, an integer $t_0$ and a constant $\xi>0$ such that for
every $w\in \hat{P}$, there exists an interval $\tilde{w}\subset
w$ with the following properties:

(1) there exists an integer $t\leq t_0$ such that
$f^{\hat{p}(w)+t}$ maps $\tilde{w}$ diffeomorphismlly onto
$\Omega$,

(2) $|\tilde{w}|\geq \xi|w|$,

(3) both components of $f^{\hat{p}(w)}(w\setminus \tilde{w})$ have
length $\geq \delta'/3$.
\end{lemma}

\begin{proof}
By Theorem 1, the preimages of each critical point $c\in
\mathcal{C}_c$ are dense in $X$. Then there exists a finite
integer $t_0>1$ such that $X\setminus\cup_{i=0}^{t_0}f^{-i}(c)$
contains no interval of length greater than $\delta'$. Let
$x:=f^{-t}(c)$ for some $0<t\leq t_0$, we can choose $w_x$ be a
maximal interval containing $x$ such that $w_x$ diffeomorphismly
onto a neighborhood $\Omega$ of $c$ and  $x$ in $w_x$ middle
fifth. By adjusting the size of $w_x$, we can make sure that they
all map onto exactly the same critical neighborhood $\Omega$ and
$w_x\leq {\delta'}/{10}$. Let $\tilde{w}\subset w\subset \hat{P}$
be the interval that is mapped onto $w_x$ by $f^{\hat{p}(w)}$,
such a $\tilde{w}$ exists by Proposition 2. Then the first and the
third statement follow. Because $f^{\hat{p}(w)}$ has bounded
distortion on $w$ by Lemma \ref{bound distortion 1}, and the
diffeomorphism $f^t$ ($t\leq t_0$) doesn't effect the distortion
significantly, the second statement follows immediately.
\end{proof}

We have fixed $\Omega$ and let $\delta''=\min\{\delta'/3,
|\Omega|\}$. We shall define a full Markov map by the construction
in Proposition 2 and Lemma \ref{full Markov}.

Let $\hat{f}:\Omega \rightarrow \Omega$ be defined as follows,
$\mathcal{Q}$ be the associated partition of $\Omega$ and $R$ be
the induced time function. For each element $w$ in the partition
$\hat{P}$ of $\Omega$, let $\tilde{w}$ denote the subinterval
given in Lemma \ref{full Markov}. We put $\tilde{w}$ into
$\mathcal{Q}$ and define $R(\tilde{w})=\hat{p}(w)+t$. Then
$\hat{f}({\tilde{w}})=f^{R(\tilde{w})}(\tilde{w})=\Omega$ by Lemma
\ref{full Markov}. On the other hand, both components of
$f^{\hat{p}(w)}(w)\backslash f^{\hat{p}(w)}(\tilde{w})$ has size
at least $\delta'/3$ ($\geq \delta''$), we consider them as new
starting intervals and carry out the construction in Section 4 and
repeat the procedure as above. In this way, for each $w\in
\mathcal{Q}$, we can define an associated sequence before it has
reached full return (i.e., $\exists n>0, f^n(w)=\Omega$), write
$\hat{p}_1=\hat{p}(x)$ and
$\hat{p}_{i+1}(x)=\hat{p}(f^{\hat{p}_i(x)}(x))$, $1<i<s-1$, where
$s$ is the integer such that $f^s$ maps corresponding element
containing $x$ to reach full return. Then we have
$R(w)=\hat{p}_s(w)+t$ for $s\geq 1, t\leq t_0$.

Next we collect some important results on the induced full Markov
map and the return time estimates, they are important bounds used
in Young \cite{yo99}.
\begin{lemma}
For each $n\leq 0$ and each interval $w\in \mathcal{Q}$ on which
${\hat{f}}^n$ is continuous, the distortion of ${\hat{f}}^n$ is
uniformly bounded.
\end{lemma}
\begin{proof}
It follows from the construction of $\hat{f}$ immediately.
\end{proof}

\begin{lemma}\label{rates}
Let $f\in \mathcal{A}_1$, $\hat{f}=f^{R}$ be the induced full
Markov map defined above. Then we have following statements:

Summable case: Under no conditions on $d_n(c)$,
$$
\sum_n|\{R>n\}|< \infty.
$$

Polynomial case: If $d_n(c)\leq C n^{-\alpha}$, $C>0$, $\alpha>1$
for all $c\in \mathcal{C}_c$ and $n\geq 1$, then there exists
$\tilde{C}>0$ such that
$$
|\{R>n\}| \leq \tilde{C} n^{-\alpha}.
$$

Stretched exponential case: If $\gamma_n(c)\leq Ce^{-\beta
n^{\alpha}}$, $C>0$, $\alpha\in (0,1)$, $\beta>0$ for all $c\in
\mathcal{C}_c$ and $n\geq 1$, then for each $\tilde{\alpha}\in
(0,\alpha)$ there exist $\tilde{\beta}, \tilde{C}>0$ such that
$$
|\{R>n\}| \leq \tilde{C} e^{-\tilde{\beta}n^{\tilde{\alpha}}}.
$$

Exponential case: If $\gamma_n(c)\leq Ce^{-\beta n}$, $C>0$,
$\beta>0$ for all $c\in \mathcal{C}_c$ and $n\geq 1$, then there
exist $\tilde{\beta}, \tilde{C}>0$ such that
$$
|\{R>n\}| \leq \tilde{C} e^{-\tilde{\beta}n}.
$$
\end{lemma}
\begin{proof}
See the details of the proof of Proposition 4.1 in \cite{BLS}
\end{proof}
For $x, y \in \mathcal{Q}$, we let separation time $s(x,y)\geq 0$
be the least integer $k$ such that ${\hat{f}^k}(x)$ and
${\hat{f}^k}(y)$ belong to different elements of $\mathcal{Q}$. So
$s(x,y)=0$ if $x,y$ are in different component of $\mathcal{Q}$,
and $s(x,y)\geq 1$ if they belongs to the same component of
$\mathcal{Q}$.

\begin{lemma}\label{bound distortion 2}
There exist $\beta\in (0,1)$ and $C>0$ such that for all $w\in
\mathcal{Q}$ and all $x,y\in w$, we have
$$
|\frac{D\hat{f}(x)}{D\hat{f}(y)}-1|\leq C{\beta}^{s(x,y)}.
$$
\end{lemma}
\begin{proof}
See \cite{BLS}.
\end{proof}

\section{Proof of Theorem 2 and Theorem 3}
For $f\in \mathcal{A}_1$, we focus on a minimal cycle of $f$. Let
it be $X=\cup_{i=0}^{m-1}f^i(J)$, where $m$ is the period of the
minimal cycle $X$. In previous Section, we have constructed an
induced full Markov map $f^{R(w)}$ on a neighborhood $\Omega$ of a
critical point $c\in J$. Then denote $g=:\Lambda^{-1}\circ f^m
\circ\Lambda$ be a renormalization of $f$, where $\Lambda$ is an
affine transformation from interval $M$ to $J$. We will apply
results of Young \cite{yo99} to study the statistical properties
of $g$ stated in Theorem 2.

From the definition of $g$, $g$ induces a Markov map
$\hat{G}:=\Lambda^{-1}\circ \hat{f} \circ\Lambda$ on the interval
$\bigtriangleup_0=\Lambda^{-1}(\Omega)$. Assume that the induced
time of $\hat{G}$ about $g$ is $T(w)$, clearly $T(w)=R(w)/m$.

To treat the statistical properties of $g$, we define a tower
about $g$,
$$
\hat{\bigtriangleup}=\{(x,k)\in \bigtriangleup_0\times Z; 0\leq
k\leq T(x)\}
$$
consists of  levels $\bigtriangleup_n = \{(x,k)\in
\hat{\bigtriangleup}; k=n\}$. We then define the map
$\hat{g}:\hat{\bigtriangleup} \rightarrow \hat{\bigtriangleup}$ by

\begin{equation*}
\hat{g}(x,k):=
\begin{cases} (x,k+1) \ \ & \ if \  k+1 \leq T(x), \ x\in w,
\\
(\hat{G}(x),0) \ \ & \  if \ k+1=T(x), \ x\in w.
   \end{cases}
   \end{equation*}

Let $\sigma:\hat{\bigtriangleup} \rightarrow M$ by $\sigma(x,
k)=g^k(x)$, which is a semi-conjugacy between $\hat{g}$ and $g$.
So we can get the statistical properties of $g$ from the
properties of $\hat{g}$.

To study the structure of $\hat{g}$, we define $\pi$ be a
projection between $\hat{\bigtriangleup}$ and $\bigtriangleup_0$
by $\pi(x,l)=x$, clearly $\pi \hat{g}=\hat{G}\pi$. The powerful
result of \cite{yo99} relates the tower map $\hat{g}$ and the full
induced Markov map $\hat{G}$.

We summarize Young's result from \cite{yo99} as we need. Recall
some notations from above, for a fixed $\beta\in (0,1)$, let
$$
\mathcal{C}_{\beta}=\{\hat{\varphi}: \hat{\bigtriangleup}
\rightarrow R; \  \exists C>0, \forall x,y \in
\hat{\bigtriangleup},\ |\hat{\varphi}(x)- \hat{\varphi}(y)|\leq
C{\beta}^{s(x,y)}\},
$$
and
$$
{\mathcal{C}_{\beta}}^+=\{\hat{\varphi} \in \mathcal{C}_{\beta};
\hat{\varphi}> 0 \}.
$$

\begin{theorem}\cite{yo99}
Suppose that $\hat{G}:\bigtriangleup_0 \rightarrow
\bigtriangleup_0$  and $\hat{g}:\bigtriangleup \to \bigtriangleup
$ defined as above.  Let $\rho_n$ be a sequence of positive
numbers related to the tail behavior as follows. If $|\{T>n\}|\leq
n^{-\alpha}$, then $\rho_n=n^{1-\alpha}$. If $|\{T>n\}|\leq
e^{-\beta n}$, then $\rho_n=e^{-\beta'n}$ for some $\beta'<\beta$.
If $|\{T>n\}|\leq e^{-n^{\alpha}}$ for some $\alpha\in (0,1)$,
then $\rho_n=e^{-n^{\alpha'}}$ for some $\alpha'<\alpha$. If
$$
|\frac{D\hat{G}(x)}{D\hat{G}(y)}-1|\leq C \beta^{s(x,y)}
$$
for any $x,y\in \bigtriangleup_0$, some $\beta\in (0,1)$ and
$C>0$,  and the induce time function $T$ satisfies
$\sum_n|T>n|<\infty$,

then we have

(1)$\hat{\bigtriangleup}$ carries a $\hat{g}$ acip $\hat{\nu}$ and
$\frac{d\hat{\nu}}{dm_{\hat{\bigtriangleup}}}\in
{\mathcal{C}_{\beta}}^+$, where $m_{\hat{\bigtriangleup}}$ be
Lebsegue measure on the tower. $(\hat{g},\hat{\nu})$ is exact,
hence ergodic and mixing.

(2)For any pair of functions $\hat{\varphi} \in
L^{\infty}(\hat{\bigtriangleup}, m_{\hat{\bigtriangleup}})$ and
$\hat{\psi} \in \mathcal{C}_{\beta}$, there exists
$C_{\hat{\varphi},\hat{\psi}}$ such that
$$
|\int (\hat{\varphi}\circ \hat{g})\hat{\psi} d\hat{\nu} -
\int\hat{\varphi} d\hat{\nu} \int\hat{\psi} d\hat{\nu}|\leq
C_{\hat{\varphi},\hat{\psi}}\rho_n.
$$

(3)If $|\{T>n\}|\leq O(n^{-\alpha})$ for some $\alpha>2$, then for
any $\hat{\varphi}\in \mathcal{C}_{\beta}$ which is not a
coboundary $(\hat{\varphi}\neq \hat{\psi}\circ
\hat{g}-\hat{\psi})$, the Central Limit Theorem holds, i.e., there
exists $\sigma> 0$ such that
$\frac{1}{\sqrt{n}}\sum_{i=0}^{n-1}\hat{\varphi}\circ {\hat{g}}^i$
converges to the normal distribution
$\mathcal{N}(\int\hat{\varphi} d\hat{\nu}, \sigma)$.

\end{theorem}

At first, Lemma \ref{rates} and  Lemma \ref{bound distortion 2}
imply that conditions in Young's Theorem hold.  If we define a
measure on $M$ by $\mu=\sigma_{*}\hat{\nu}$, i.e,
$\mu(B)=\hat{\nu}(\sigma^{-1}(B))$ for all measurable set $B$, we
have
\begin{equation*}
\begin{split}
\mu (g^{-1}(B))&=\hat{\nu}(\sigma^{-1}(g^{-1}(B)))=
\hat{\nu}(\hat{g}^{-1}(\sigma^{-1}(B)))\\
&=\hat{\nu}(\sigma^{-1}(B))=\mu(B).
\end{split}
\end{equation*}
So $\mu$ is an invariant measure of $g$. Note that the reference
measure $m_{\hat{\bigtriangleup}}$ defined by
$m_{\hat{\bigtriangleup}}(A)=\sum_{k\geq 0}m(\pi(A\cap
\bigtriangleup_k))$ for all measurable set $A\in
\hat{\bigtriangleup}$. Let $I_k=\sigma(\bigtriangleup_k)$, for any
Borel set $B\in M$, there exists a constant $C>0$ such that
\begin{equation*}
\begin{split}
\mu(B)=\hat{\nu}({\sigma^{-1}(B)}) &\leq C
m_{\hat{\bigtriangleup}}(\sigma^{-1}(B))=C\sum_{k\geq
0}m(\pi(\sigma^{-1}(B)\cap \bigtriangleup_k)) \\
&\leq C \sum_{k\geq 0} m(B\cap I_k) =Cm(B),
\end{split}
\end{equation*}
where the last inequality follows from the bounded distortion of
$\sigma$(since $g^i$ has bounded distortion for each $w\in
\bigtriangleup$ and $i<T(w)$). So $\mu$ is an acip of $g$. On the
other hand, because $\sigma$ has bounded distortion,
$$
\frac{d\mu}{dm}=\frac{d\sigma_*\nu}{dm}>0
$$
from above theorem. Then the support is equal to the entire
interval $M$, which in turn implies that $\mu$ is the unique acip
for $g$ on $M$. According Theorem 1, if $l_{\max}>1$ for all $c$
in the cycle $X$, we obtains that $\frac{d\mu}{dm}\in L^{\tau}$
for all $0<\tau<\frac{l_{\max}}{l_{\max}-1}$.

We also obtain that $(g,\mu)$ is exact. Indeed, if there are sets
$B_0,B_1,\dots,B_j\subset M$ such that $\mu(B_j)\in (0,1)$ and
$B_0=g^{-j}(B_j)$ for all integer $j\geq 0$, then let
$A_j={\sigma}^{-1}(B_j)$, $\hat{\nu}(A_j)\in(0,1)$, therefore
$$
\hat{g}^{-j}(A_j)={\hat{g}}^{-j}({\sigma}^{-1}(B_j))={\sigma}^{-1}
{g}^{-j}(B_j)= {\sigma}^{-1}(B_0)=A_0,
$$
which contradicts to the first statement of Young's Theorem.

We shall consider the decay of correlations of $(g,\mu)$. Let
$\varphi, \psi:M\rightarrow R$ be two H\"{o}lder continuous
function so that there exist $C>0$ and $\alpha\in (0,1)$
$$
|\varphi(x_1)-\varphi(x_2)|\leq  C|x_1-x_2|^{\alpha},
$$
for all $x_1,x_2\in M$. Let $\hat{\varphi}=\varphi\circ \sigma$,
then if $x,y \in \bigtriangleup_0$ with $s(x,y)\geq 1$, we have
for $0<k\leq T(x)=T(y)$
\begin{equation*}
\begin{split}
|\hat{\varphi}(x,k)-\hat{\varphi}(y,k)|&=|\varphi \circ
g^k(x)-\varphi \circ g^k(y)| \leq C{|g^k(x)-g^k(y)|}^{\alpha}
\\
&\leq C{|\hat{f}(x)-\hat{f}(y)|}^{\alpha} \leq C
{({\beta}^\alpha)}^{s(x,y)},
\end{split}
\end{equation*}
thus $\alpha$-H\"{o}lder continuous observables  on $M$ correspond
to observables in $\mathcal{C}_{\beta_1}$ for
$\beta_1={\beta}^{\alpha}$, where $\beta$ is the constant in Lemma
\ref{bound distortion 2}.

Since $\mu=\sigma_*\hat{\nu}$, $\hat{\varphi}=\varphi\circ \sigma$
and  $\hat{\psi}=\psi\circ \sigma$, we have the following
relations $\int \hat{\varphi} d\hat{\nu}=\int \varphi d\mu$, $\int
\hat{\psi} d\hat{\nu}=\int \psi d\mu$ and
\begin{equation*}
|\int (\varphi \circ g^n)\psi d\mu - \int\varphi d\mu \int\psi
d\mu|= |\int (\hat{\varphi} \circ \hat{g}^n) \hat{\psi} d\hat{\nu}
- \int \hat{\varphi} d\hat{\nu} \int \hat{\psi} d\hat{\nu}|.
\end{equation*}
Thus the decay of correlations of $g$ follows from the second
statement of the Young's Theorem.

Similarly, because
\begin{equation*}
\mu \Big{\{} x\in M; \big{(}\frac{1}{\sqrt{n}} \sum_{j=0}^{n-1}
\varphi (g^j(x))- \int \varphi d\mu \big{)} \in I \Big{\}}  =
\hat{\nu} \Big{\{} y\in \hat{\bigtriangleup};
\big{(}\frac{1}{\sqrt{n}} \sum_{j=0}^{n-1} \hat{\varphi}
({\hat{g}}^j(y))- \int \hat{\varphi} d\hat{\nu} \big{)} \in I
\Big{\}},
\end{equation*}
applying the third statement of Young's Theorem, we obtain the
Central Limit Theorem for the H\"{o}lder continuous observable
$\varphi$. Therefore, we have finished the proof of Theorem 2.

For $f\in \mathcal{A}_2$, we concentrated on the dynamics of $(f,
J^*)$. In previous sections, we have constructed an induced full
Markov map $f^{R(w)}$ on a neighborhood $\Omega$ of a critical
point $c\in J^*$. Let $k$ be the greatest common devisor of all
values taken by the function $R(w)$, then we can obtain Theorem 3
by using Young's results and the above argument similarly.

\end{document}